\documentclass[reqno]{amsart}
\usepackage{amscd,amssymb,amsmath,amsfonts,amsthm,latexsym}
\usepackage{extpfeil}
\usepackage[latin1]{inputenc}
\usepackage[all]{xy}

\def \Zenter{\mathop{\sf Z}\nolimits}

\newcommand{\cero}{\bar{0}}
\newcommand{\scero}{_{\bar{0}}}
\newcommand{\uno}{\bar{1}}
\newcommand{\suno}{_{\bar{1}}}
\newcommand{\menosuno}[2]{(-1)^{|#1| |#2|}}
\newcommand{\sh}{^{\sharp}}

\newcommand{\slmn}{\mathfrak{sl}(m, n, A)}
\newcommand{\stmn}{\mathfrak{st}(m, n, A)}
\newcommand{\stdu}{\mathfrak{st}(2, 1, A)}
\newcommand{\sttu}{\mathfrak{st}(3, 1, A)}
\newcommand{\stdd}{\mathfrak{st}(2, 2, A)}

\newcommand{\Hij}{H_{ij}}

\newcommand{\Hji}{H_{ji}}
\newcommand{\Huj}{H_{1j}}
\newcommand{\thij}{\widetilde{H}_{ij}}
\newcommand{\thkl}{\widetilde{H}_{kl}}
\newcommand{\thik}{\widetilde{H}_{ik}}
\newcommand{\thji}{\widetilde{H}_{ji}}

\newcommand{\Fij}{F_{ij}}
\newcommand{\Fji}{F_{ji}}
\newcommand{\Fik}{F_{ik}}
\newcommand{\Fkj}{F_{kj}}
\newcommand{\Fkl}{F_{kl}}
\newcommand{\Fki}{F_{ki}}
\newcommand{\Fjk}{F_{jk}}
\newcommand{\Flj}{F_{lj}}
\newcommand{\Fil}{F_{il}}

\newcommand{\tfij}{\widetilde{F}_{ij}}
\newcommand{\tfji}{\widetilde{F}_{ji}}
\newcommand{\tfik}{\widetilde{F}_{ik}}
\newcommand{\tfkj}{\widetilde{F}_{kj}}
\newcommand{\tfkl}{\widetilde{F}_{kl}}
\newcommand{\tfki}{\widetilde{F}_{ki}}

\newcommand{\tflj}{\widetilde{F}_{lj}}
\newcommand{\tfil}{\widetilde{F}_{il}}
\newcommand{\tfjl}{\widetilde{F}_{jl}}

\DeclareMathOperator{\Ho}{H}

\DeclareMathOperator{\hc}{HC}

\DeclareMathOperator{\VV}{\mathcal{V}}

\DeclareMathOperator{\Str}{Str}

\DeclareMathOperator{\III}{\mathcal{I}}
\DeclareMathOperator{\WW}{\mathcal{W}}

\newtheorem{Th}{Theorem}[section]

\newtheorem{Le}[Th]{Lemma}
\newtheorem{Co}[Th]{Corollary}
\theoremstyle{definition}
\newtheorem{De}[Th]{Definition}

\theoremstyle{remark}

\begin{document}

\title[Universal central extensions of $\mathfrak{sl}(m, n, A)$ of small rank]{Universal central extensions of $\mathfrak{sl}(m, n, A)$ of small rank over associative superalgebras}
\author{X. García-Martínez}
\address{[X. García-Martínez] Department of Algebra, University of Santiago de Compostela, 15782, Spain.}
\email{xabier.garcia@usc.es}
\author{M. Ladra}
\address{[M. Ladra] Department of Algebra, University of Santiago de Compostela, 15782, Spain.}
\email{manuel.ladra@usc.es}

\thanks{The authors were supported by Ministerio de Economía y Competitividad (Spain), grant MTM2013-43687-P (European FEDER support included).
The authors were also supported by Xunta de Galicia, grant GRC2013-045 (European FEDER support included).
The first author was also supported  by FPU scholarship, Ministerio de Educación, Cultura y Deporte  (Spain).}

\begin{abstract}
We complete the solution of the problem of finding the universal central extension of the matrix superalgebras $\slmn$ where $A$ is an associative superalgebra and computing $\Ho_2\big(\slmn\big)$. The Steinberg Lie superalgebra $\stmn$ has a very important role and we will also find out $\Ho_2\big(\stmn\big)$. In \cite{ChSu} it is solved the problem where $m+n \geq 5$ and in \cite{ChGu} it is solved when $n=0$, so here we work out the three remaining cases $\mathfrak{sl}(2, 1, A), \mathfrak{sl}(3, 1, A)$ and $\mathfrak{sl}(2, 2, A)$.

\end{abstract}

\maketitle

\section{Introduction}
The study of central extensions plays an important role in the theory of groups or Lie algebras. This theory has numerous applications going through physics, representation theory or homological algebra. They have been studied by many people in the theory of Lie algebras as \cite{Gar}, \cite{Van}, etc. The universal central extension performs a relevant role in this study, since it simplifies the task of finding all central extensions and moreover, its kernel is the second homology group. In \cite{Ell1} the universal central extension of Lie algebras is constructed as a non-abelian tensor product, and in \cite{Neh} and \cite{IoKo1} some of the results of \cite{Gar} are extended to Lie superalgebras and it is constructed the universal central extension. The main problem of these constructions is that they are usually hard to compute.

The concrete problem of finding the universal central extension of $\mathfrak{sl}_n(A)$ for $n \geq 5$ was solved in \cite{KaLo} and it is a very important result, since it involves Steinberg Lie algebras, which have been very studied (see \cite{Blo} or \cite{Gao}) and they allow to see the additive $K$-theory. It is obtained that if $n \geq 5$, then $\mathfrak{st}_n(A)$ is the universal central extension of $\mathfrak{sl}_n(A)$ and if $A$ is $K$-free, the kernel is isomorphic to the first cyclic homology $\hc_1(A)$. The problem of finding the universal central extension of $\mathfrak{sl}_n(A)$ and $\mathfrak{st}_n(A)$ for $n = 3, 4$ was solved years later in \cite{GaSh}.

Moreover, in \cite{MiPi} it is computed the universal central extension of the Lie superalgebras $\slmn$ and $\stmn$ with $m+n \geq 5$, where $A$ is an associative algebra, and the remaining cases where $m+n = 3, 4$ are solved in \cite{SCG}.

If $A$ is an associative superalgebra, the universal central extension of $\mathfrak{sl}_n(A)$ is computed in \cite{ChGu} for all $n \geq 3$, and for $\slmn$ for $m+ n \geq 5$ in \cite{ChSu}, leaving as an open problem the cases where $m+n=3, 4$. In this paper, we will solve this remaining cases completing the computation of the universal central extension of $\slmn$ and $\stmn$ with $A$ an associative superalgebra and $m+n \geq 3$, and therefore giving a complete characterization of the second homology $\Ho_2\big(\stmn\big)$ and $\Ho_2\big(\slmn\big)$ for $m+n \geq 3$ (Theorem \ref{T:con1} and Theorem \ref{T:con2})

The organization of this paper is as follows. After this introduction, in Section 2 we give some preliminary well-known results and some technical lemmas about $\slmn$ and $\stmn$. In Section 3 we adapt the classical construction of a central extension from a super 2-cocycle in Lie superalgebras. In Section 4 we start with the case of $\mathfrak{sl}(2, 1, A)$ and we show that its universal central extension is $\stdu$, building a (unique) homomorphism to any central extension. In Section 5 we find the universal central extension of $\sttu$ (which consequently will be the universal central extension of $\mathfrak{sl}(3, 1, A)$) via the construction of a super 2-cocycle. In Section 6 we do the same as in the previous section but for $\stdd$. Finally, in Section 7 we give a concluding remarks establishing a combination of the results presented here with results of \cite{ChSu} and \cite{ChGu} to give the full computation of $\Ho_2\big(\stmn\big)$ and $\Ho_2\big(\slmn\big)$ for $m+n \geq 3$.

\section{The Lie superalgebras $\slmn$ and $\stmn$}

Through all this paper we will consider $K$ as an unital commutative ring and $A = A\scero \oplus A\suno$ an associative unital $K$-superalgebra, $K$-free with a $K$-basis containing the identity. Let $\{ 1, \dots ,m \} \cup \{ m+1, \dots ,m+n \}$ be the graded set, where the first set is the even part and the second one the odd part. We consider now $\text{Mat}(m, n, A)$ the $(m+n) \times (m+n)$ matrices with coefficients in $A$. It is defined a graduation where homogeneous elements are matrices, denoted by $E_{ij}(a)$, having $a \in A\scero, A\suno$ at position $(i, j)$ and zero elsewhere and $|E_{ij}(a)| = |i| + |j| + |a|$. With this graduation we define the associative superalgebra $\mathfrak{gl}(m, n, A)$ which underlying set is $\text{Mat}(m, n, A)$ with the usual matrix product and it is endowed by a Lie superalgebra structure with the usual bracket $[x, y] = xy - \menosuno{x}{y}yx$.

Assuming that $m+n \geq 3$ we define the Lie superalgebra
\[
\slmn = [\mathfrak{gl}(m, n, A), \mathfrak{gl}(m, n, A)].
\]

In \cite{ChSu} it is shown that $\slmn$ is the Lie superalgebra generated by the elements $E_{ij}, 1 \leq i \neq j \leq m+n$, $a \in A$, where the Lie bracket is given by
\[
[E_{ij}(a), E_{kl}(b)] = \delta_{jk}E_{il}(ab) - \menosuno{E_{ij}(a)}{E_{kl}(ab)} \delta_{li}E_{kj}(ba).
\]

In \cite{Bla} is introduced a generalization of the supertrace for $x \in \mathfrak{gl}(m, n, A)$, defined as follows:
\[
\Str(x) = \sum_{i = 1}^{m+n}(-1)^{|i|(|i|+|x_{ii}|)}x_{ii},
\]
where $x_{ii}$ represents the element of $x$ in the position $(i,i)$. It is shown in \cite{ChSu} that if $m \geq 1$ then $\slmn = \{ x \in \mathfrak{gl}(m, n, A) : \Str(x) \in [A, A] \}$ and that $\slmn$ is perfect.

For $m + n \geq 3$ the Steinberg Lie superalgebra $\stmn$ is defined as the Lie superalgebra over $K$ generated by homogeneous $F_{ij}(a)$, $1 \leq i \neq j \leq m+n$ and $a \in A$ an homogeneous element, where the grading is given by $|F_{ij}(a)| = |i| + |j| + |a|$, satisfying the following relations:
\begin{equation}\label{Eq:st1}
a \mapsto F_{ij}(a) \text{ is a }K\text{-linear map,}
\end{equation}
\begin{equation}\label{Eq:st2}
[F_{ij}(a), F_{jk}(b)] = F_{ik}(ab), \text{ for distinct }i, j, k,
\end{equation}
\begin{equation}\label{Eq:st3}
[F_{ij}(a), F_{kl}(b)] = 0, \text{ for } j \neq k, i \neq l,
\end{equation}
where $a, b \in A$, $1 \leq i,j,k,l \leq m+n$. We know that $\stmn$ is a perfect Lie algebra and there is a canonical epimorphism
\[
\varphi \colon \stmn \to \slmn, \qquad \varphi\big(F_{ij}(a)\big) \mapsto E_{ij}(a).
\]

In \cite{ChSu} it is shown that if $m+n \geq 5$, this epimorphism is the universal central extension of $\slmn$ and if $m + n \geq 3$, its kernel is isomorphic to the cyclic homology $\hc_1(A)$. The remaining cases where $m + n = 3, 4$ are left as an open problem and they will be the object of study of this work. To solve this we will find the universal central extension of $\stmn$ and by \cite[Corollary 1.9]{Neh} it is the universal central extension of $\slmn$.

We begin giving some relations in $\stmn$ that will be useful. Let be
\[
\Hij(a, b) = [\Fij(a), \Fji(b)],
\]
\[
h(a, b) = \Huj(a,b) - \menosuno{a}{b}\Huj(1, ba),
\]
for $1 \leq i \neq j \leq m+n, a \in A$. It is well defined since by \cite[Lemma 5.1]{ChSu} we know that $h(a, b)$ does not depend on $j$, for $j \neq 1$. We recall that $|\Hij(a, b)| = |a|+|b|$ for homogeneous $a, b \in A$.

\begin{Le}\label{L:ident}
We have the following identities in $\stmn$,
\begin{align}
{}& \Hij(a, b) = - (-1)^{(|i| + |j| + |a|)(|i| + |j| + |b|)}\Hji(b, a), \label{Eq:id1}\\
{}& [\Hij(a, b), \Fik(c)] = \Fik(abc),\label{Eq:id2} \\
{}& [\Hij(a, b), \Fki(c)] = -(-1)^{(|a| + |b|)(|i| + |k| + |c|)} \Fki(cab), \label{Eq:id3} \\
{}&[\Hij(a, b), \Fkj(c)] = (-1)^{(|i|+|j|+|a|)(|i|+|j|+|b|) + (|a|+|b|)(|j|+|k|+|c|)}\Fkj(cba) \label{Eq:id4} \\
{}& [\Hij(a, b), \Fij(c)] = \Fij\big(abc + (-1)^{(|i| + |j| + |a||b| + |b||c| + |c||a|)}cba \big), \label{Eq:id5} \\
{}& [\Hij(a, b), \Fkl(c)] = 0, \label{Eq:id6} \\
{}& [h(a, b), F_{1i}(c)] = F_{1i}\big((ab - \menosuno{a}{b}ba)c  \big), \label{Eq:id7}   \\
{}& [h(a, b), F_{jk}(c)] = 0 \text{ for } j, k \geq 2. \label{Eq:id8}
\end{align}
for homogeneous $a, b, c \in A$ and $i, j, k, l$ distinct.
\end{Le}

\begin{proof}
The proofs of \eqref{Eq:id1}--\eqref{Eq:id6} can be found in \cite{ChSu}. They are consequences of antisymmetry and Jacobi identities. To see \eqref{Eq:id7} and \eqref{Eq:id8} we just need to apply \eqref{Eq:id2} and \eqref{Eq:id6} to the definition of $h(a, b)$.
\end{proof}

Following the steps of \cite{ChSu} we have a lemma that simplifies the structure of $\stmn$.
\begin{Le}\label{L:dec}
Let $F_{ij}(A)$ be the subalgebra generated by $\Fij(a)$, $\mathcal{N}^{+}$ the subalgebra generated by $\Fij(a)$ for $1 \leq i < j \leq m+n$, $\mathcal{N}^{-}$ the subalgebra generated by $\Fij(a)$ for $1 \leq j < i \leq m+n$ and $\mathcal{H}$ the subalgebra generated by $\Hij(a, b)$, for all $a, b \in A$. Then
\[
\mathcal{N}^{+} = \bigoplus_{1 \leq i < j \leq m+n} \Fij(A),
\]
\[
\mathcal{N}^{-} = \bigoplus_{1 \leq j < i \leq m+n} \Fij(A),
\]
\[
\mathcal{H} = h(A, A) \oplus \bigg( \bigoplus_{j = 2}^{m+n}\Huj(1, A) \bigg),
\]
and we have the decomposition
\[
\stmn = \mathcal{N}^{+} \oplus \mathcal{H} \oplus \mathcal{N}^{-} = h(A, A) \oplus \bigg( \bigoplus_{j = 2}^{m+n}\Huj(1, A) \bigg) \bigoplus_{1 \leq i \neq j \leq m+ n} \Fij(A),
\]
where $h(A, A)$ and $H_{1j}(1, A)$ denote the subalgebras generated by $h(a, b)$ and $H_{1j}(1, a)$, respectively, for all $a, b \in A$.
\end{Le}

To understand the universal central extension of $\stmn$ we will also need the next definition.

\begin{De}\label{D:Am}
For $m \geq0$, we define $\III_m$ as the graded ideal of $A$ generated by the elements $ma$ and $ab - \menosuno{a}{b}ba$. We will denote by $A_m =  A/\III_m$ the quotient algebra and by $\bar{a} = a + \III_m$ its elements.
\end{De}

\begin{Le}[\cite{ChGu}]
$\III_m = mA + A[A, A]$ and $[A, A]A = A[A, A]$.
\end{Le}

\section{Central extensions of $\slmn$ and cocycles}

\begin{De}
Let $L$ a Lie superalgebra and $\WW$ a $K$-free supermodule. A \emph{super 2-cocycle} is a $K$-bilinear map $\psi \colon L \times L \to \WW$ satisfying that
\begin{gather*}
\psi (x, y) = -\menosuno{x}{y} \psi(y, x), \\
\menosuno{x}{z} \psi([x, y], z) + \menosuno{x}{y} \psi([y, z], x) + \menosuno{y}{z}\psi([z, x], y) = 0, \\
\psi(x\scero, x\scero) = 0,
\end{gather*}
for all $x, y, z \in L$ and $x\scero \in L\scero$.
\end{De}

Given an even super 2-cocycle $\psi$, we can construct a central extension as we can see in \cite{Neh}. This central extension will be $L \oplus \WW$ with bracket $[(x, w_1), (y, w_2)] = \big([x, y], \psi(x, y)\big)$ and the morphism is given by the projection to $L$. We will see that if $L = \stmn$ and the super 2-cocycle is surjective, this construction can be described in a different way using relations.

\begin{De}
Let $\psi \colon \stmn \times \stmn \to \WW$ be an even super 2-cocycle, i.e. a super 2-cocycle such that $|\psi(x, y)| = |x|+|y|$ for homogeneous $x, y \in \stmn$. Let $\stmn\sh$ be the Lie superalgebra generated by the elements $\Fij(a)\sh$ with homogeneous $a \in A$ and $1 \leq i \neq j \leq m+n$ with degree $|\Fij\sh(a)| = |i| + |j| + |a|$ and by the elements of $\WW$, with the relations
\begin{align*}
{}& a \mapsto \Fij\sh(a) \text{ is a }K\text{-linear map,} \\
{}& [\WW, \WW] = [\Fij\sh(a), \WW] = 0, \\
{}& [\Fij\sh(a), \Fjk\sh(b)] = \Fik\sh(ab) + \psi\big( \Fij(a), \Fjk(b) \big) \text{ for distinct i, j, k,} \\
{}& [\Fij\sh(a), \Fkl\sh(b)] = \psi\big( \Fij(a), \Fkl(b) \big) \text{ for } i \neq j \neq k \neq l \neq i,
\end{align*}
where $a, b \in A$.
\end{De}

\begin{Le}
If $\stmn' = \stmn \oplus \WW$ is a central extension constructed from a surjective super 2-cocycle $\psi \colon \stmn \times \stmn \to \WW$ then there is an isomorphism $\rho \colon \stmn\sh \to \stmn'$ where $\rho\big(\Fij\sh(a)\big) = \Fij(a)$ and $\rho(w) = w$.
\end{Le}

\begin{proof}
The proof is the same as in \cite[Lemma 1]{ChGu} with the only difference that we work with $\stmn$ and not with $\mathfrak{st}(n, 0, A)$, but adding the superalgebra structure does not make any difference.
\end{proof}

As we did before, we denote $\Hij\sh(a, b) = [\Fij\sh(a), \Fji\sh(b)]$ and $h\sh(a, b) = \Huj\sh(a, b) - \menosuno{a}{b} \Huj\sh(1, ba)$. We know that $h\sh$ is independent of $j$ and we have the analogue decomposition lemma.
\begin{Le}
We can decompose the Lie superalgebra $\stmn\sh$ generated by a surjective super 2-cocycle $\psi \colon \stmn \times \stmn \to \WW$ in the following way:
\[
\stmn\sh = \WW \ \oplus \ h\sh(A, A) \oplus \bigg( \bigoplus_{j = 2}^{m+n}\Huj\sh(1, A) \bigg) \bigoplus_{1 \leq i \neq j \leq m+ n} \Fij\sh(A),
\]
where $\Fij\sh(A)$, $h\sh(A, A)$ and $H_{1j}\sh(1, A)$ denote the subalgebras generated by $\Fij\sh(a)$, $h\sh(a, b)$ and $H_{1j}\sh(1, a)$, respectively,  for all $a, b \in A$.
\end{Le}

\begin{proof}
The proof is the same as for Lemma \ref{L:dec}.
\end{proof}

\section{Universal central extension of $\stdu$}

First of all we study the case when $m + n = 3$ and see that $\stdu$ is itself the universal central extension of $\mathfrak{sl}(2, 1, A)$.

\begin{Th}\label{T:du}
If $\tau \colon \widetilde{\mathfrak{st}}(2, 1, A) \to \stdu$ is a central extension, then there exists a unique section $\eta \colon \stdu \to \widetilde{\mathfrak{st}}(2, 1, A)$.
\end{Th}

\begin{proof}
To obtain this result we will directly construct a Lie superalgebra homomorphism $\eta \colon \stdu \to \widetilde{\mathfrak{st}}(2, 1, A)$, such that $\tau \circ \eta = \text{id}$ and since $\stdu$ is perfect it is unique.
Let
\[
\xymatrix{
0 \ar[r] & \VV \ar[r] & \widetilde{\mathfrak{st}}(2, 1, A) \ar[r]^\tau & \stdu \ar[r] & 0
}
\]
be a central extension. For any $\Fij(a)$ where $a \in A$ is in the $K$-basis, we choose a preimage denoted by $\tfij(a)$ and extend it by $K$-linearity to all $a \in A$.

We define $\thij(a, b) = [\tfij(a), \tfji(b)]$, since it is independent of the choice of $\tfij(a)$.
By identity \eqref{Eq:id1} we know that $[\thik(1, 1), \tfij(a)] = \tfij(a) + v_{ij}(a)$, where $v_{ij}(a) \in \VV$ so we will replace $\tfij(a)$ by $\tfij(a) + v_{ij}(a)$.

Let us see that these $\tfij(a)$ follow the relations \eqref{Eq:st1}--\eqref{Eq:st3} of the definition of the Steinberg Lie superalgebra. If it is true, our $K$-linear section $\eta \colon \stdu \to \widetilde{\mathfrak{st}}(2, 1, A)$, $\Fij(a) \mapsto \tfij(a)$ will be a Lie superalgebra homomorphism and the result is proved. The first relation is immediate by definition.

To see the second one, we use the Jacobi identity and the fact that $\VV$ is in the centre of $\widetilde{\mathfrak{st}}(2, 1, A)$.
\begin{align*}
\tfij(ab) &= [\thik(1, 1), \tfij(ab)] = \big[\thik(1, 1), [\tfik(a), \tfkj(b)]\big] \\
        {}&= \big[[\thik(1, 1), \tfik(a)], \tfkj(b)\big] + \big[\tfik(a), [\thik(1, 1), \tfkj(b)]\big] \\
        {}&= [\tfik(a + (-1)^{|i| + |k|}a ), \tfkj(b)] + [\tfik(a), -(-1)^{(|i|+|k|)(|i|+|k|)} \tfkj(b)] \\
        {}&= [\tfik(a), \tfkj(b)].
\end{align*}

Now we will see that the remaining brackets are equal to zero.

\begin{align*}
[\tfij(a), \tfij(b)] &= \big[\tfij(a), [\tfik(b), \tfkj(1)] \big] \\
				   {}&= \big[[\tfij(a), \tfik(b)], \tfkj(1) \big] \\
				   {}&\quad \ + (-1)^{(|i|+|j|+|a|)(|i|+|k|+|b|)} \big[\tfik(b), [\tfij(a), \tfkj(1)] \big] = 0.
\end{align*}

To see $[\tfij(a), \tfik(b)]$ we suppose that $|i| + |j| = \uno$, then
\begin{align*}
0 &= (-1)^{\uno + |a|} \big[\thij(1, 1), [\tfij(a), \tfik(b)]\big] \\
{}&= (-1)^{\uno + |a|}\big[[\thij(1, 1), \tfij(a)], \tfik(b)\big] \\
{}&\quad \ + (-1)^{(\uno + |a|)+(|i|+|j|)(|i| + |j| + |a|)}\big[\tfij(a), [\thij(1, 1), \tfik(b)]\big] \\
{}&= (-1)^{\uno + |a|}[\tfij(a + (-1)^{\uno}a), \tfik(b)] + [\tfij(a), \tfik(b)] = [\tfij(a), \tfik(b)].
\end{align*}

If $|i| + |j| = \cero$, we have that $|i| + |k| = \uno$ and the calculation is the same, proving that $[\tfij(a), \tfkl(b)] = 0$ if $j \neq k$ and $i \neq l$, satisfying relation \eqref{Eq:st3} and completing the proof.
\end{proof}

\begin{Co}
The universal central extension of $\mathfrak{sl}(2, 1, A)$ and $\stdu$ is $\stdu$. Moreover, $\Ho_2\big(\stdu) = 0$.
\end{Co}

\section{Universal central extension of $\sttu$}\label{S:sttu}

We move now to the case when $m + n = 4$. We will find the universal central extension of $\sttu$ which will be the same as the universal central extension of $\mathfrak{sl}(3, 1, A)$.

Let $S_4$ the symmetric group of $\{1, 2, 3, 4\}$, i.e. the set of all quadruples $(i, j, k, l)$ where $1\leq i, j, k, l \leq 4$ distinct. We quotient $S_4$ by Klein's subgroup, formed by $\{ (1, 2, 3, 4), (3, 2, 1, 4), (1, 4, 3, 2), (3, 4, 1, 2)  \}$, obtaining 6 cosets denoted by $P_m$. We have a map $\theta$ that sends $(i, j, k, l)\mapsto \theta\big((i, j, k, l)\big) = m$ when $(i, j, k, l) \in P_m$.

Let $\Pi(A_2)$ be the $K$-supermodule $A_2$ (see Definition \ref{D:Am}) with the parity changed, i.e.,  $\big(\Pi(A_2) \big)\scero = (A_2)\suno$ and $\big(\Pi(A_2) \big)\suno = (A_2)\scero$.
Let $\WW = \Pi(A_2)^6$ be the $K$-supermodule formed by the direct sum of six copies of $\Pi(A_2)$ and let be the maps $\epsilon_m \colon \Pi(A_2) \to \WW$, $\epsilon_m(\bar{a}) \mapsto(0, \dots, \bar{a}, \dots, 0)$ in the position $m$.

Using the decomposition of Lemma \ref{L:dec} we build a $K$-bilinear map, defining it in the elements of the $K$-basis and extending it by $K$-linearity,
\[
\psi \colon \sttu \times \sttu \to \WW,
\]
where
\begin{align*}
{}&\psi\big(\Fij(a), \Fkl(b)\big) = \epsilon_{\theta\big((i, j, k, l)\big)}(\overline{ab}), \\
{}&\psi(x, y) = 0 \text{ if } x \text{ or } y \text{ belong to } \mathcal{H}.
\end{align*}

\begin{Le}
The $K$-bilinear map $\psi$ is a super 2-cocycle.
\end{Le}

\begin{proof}
Since the grading in $\WW$ is changed and one and just one index is odd, we have that
\[
|\psi\big(\Fij(a), \Fkl(b)\big)| = |i| + |j| + |a| + |k| + |l| + |b| = |a| + |b| + \uno = | \epsilon_{\theta\big((i, j, k, l)\big)}(\overline{ab})|,
\]
for homogeneous $a, b \in A$, so $\psi$ is even.

To complete the proof we can just follow the steps of \cite[Lemma 2.2]{GaSh} since $\bar{a} = -\bar{a}$ so the signs do not have any importance.
\end{proof}

By the previous lemma, we have a central extension
\[
\xymatrix{
0 \ar[r] & \WW \ar[r] & \sttu\sh \ar[r]^{\pi} & \sttu \ar[r] & 0,
}
\]
where $\sttu\sh = \sttu \oplus \WW$ is the Lie superalgebra constructed by the surjective super 2-cocycle $\psi$, defined by the following relations
\begin{align}
{}& a \mapsto \Fij\sh(a) \text{ is a }K\text{-linear map,} \label{Eq:sh1} \\
{}& [\WW, \WW] = [\Fij\sh(a), \WW] = 0, \label{Eq:sh2} \\
{}& [\Fij\sh(a), \Fjk\sh(b)] = \Fik\sh(ab) \text{ for distinct i, j, k,} \label{Eq:sh3} \\
{}& [\Fij\sh(a), \Fij\sh(a)] = 0, \label{Eq:sh4} \\
{}& [\Fij\sh(a), \Fik\sh(b)] = 0, \label{Eq:sh5} \\
{}& [\Fij\sh(a), \Fkl\sh(b)] = \epsilon_{\theta\big((i, j, k, l)\big)}(\overline{ab}) \text{ for distinct } i, j, k, l. \label{Eq:sh6}
\end{align}

\begin{Th}\label{T:sttu}
The central extension $0 \to \WW \to \sttu\sh \to \sttu$ is universal.
\end{Th}

\begin{proof}
Let
\[
\xymatrix{
0 \ar[r] & \VV \ar[r] & \widetilde{\mathfrak{st}}(3, 1, A) \ar[r]^{\tau} & \sttu \ar[r] & 0
}
\]
be a central extension. We need to show that there exists a Lie superalgebra homomorphism $\rho \colon \sttu\sh \to \widetilde{\mathfrak{st}}(3, 1, A)$ such that $\tau \circ \rho = \pi$.

We choose a preimage $\tfij(a)$ of $\Fij(a)$ as we did in Theorem \ref{T:du}, choosing preimages where $a \in A$ are the elements of the $K$-basis and extending them by $K$-linearity. Since $\VV \subset \Zenter(\widetilde{\mathfrak{st}}(3, 1, A))$, we have that
\[
[\tfik(a), \tfkj(b)] = \tfij(ab) + v_{ijk}(a, b),
\]
for distinct $i, j, k$, where $v_{ijk}(a, b) \in \VV$. Using Jacobi identity we have
\begin{align*}
[\tfik(a), \tfkj(cb)] &= \big[\tfik(a), [\tfkl(c), \tflj(b)]\big] \\
{}&= \big[[\tfik(a), \tfkl(c)], \tflj(b)\big] \\
{}&\quad \ + (-1)^{(|i|+|k|+|a|)(|k|+|l|+|c|)}\big[\tfkl(c), [\tfik(a), \tflj(b)]\big] \\
{}&= [\tfil(ac), \tflj(b)],
\end{align*}
so choosing $c = 1$ we have the identities $v_{ijk}(a, b) = v_{ijl}(a, b)$ and $[\tfik(a), \tfkj(b)] = [\tfil(a), \tflj(b)]$. This means that $v_{ijk}(a, b)$ is independent of the choice of $k$ so we have
\[
[\tfik(a), \tfkj(b)] = \tfij(ab) + v_{ij}(a, b),
\]
and
\[
[\tfik(1), \tfkj(b)] = \tfij(b) + v_{ij}(1, b).
\]

Now we replace $\tfij(b)$ by $\tfij(b) + v_{ij}(1, b)$. We want to define $\rho\big(\Fij\sh(a)\big) = \tfij(a)$ so will see that these elements satisfy relations \eqref{Eq:sh1}--\eqref{Eq:sh6}.

Relations \eqref{Eq:sh1}, \eqref{Eq:sh2} and \eqref{Eq:sh3} are straightforward by definition. To see relation \eqref{Eq:sh4}, we choose $i, j, k$ distinct
\begin{align*}
[\tfij(a), \tfij(b)] &= \big[\tfij(a),[\tfik(b), \tfkj(1)] \big] \\
{}&= \big[[\tfij(a),\tfik(b)], \tfkj(1) \big] \\
{}&\quad \ + (-1)^{(|i|+|j|+|a|)(|i|+|k|+|b|)}\big[\tfik(b),[\tfij(a), \tfkj(1)] \big] \\
{}&= 0.
\end{align*}

For relation \eqref{Eq:sh5}, taking $i, j, k, l$ distinct, we have
\begin{align*}
[\tfij(a), \tfik(b)] &= \big[\tfij(a), [\tfil(b), \tfik(1)]\big] \\
{}&= \big[[\tfij(a), \tfil(b)], \tfik(1)\big] \\
{}&\quad \ + (-1)^{(|i|+|j|+|a|)(|i|+|l|+|b|)}\big[\tfil(b), [\tfij(a), \tfik(1)]\big] \\
{}&= 0.
\end{align*}

For relation \eqref{Eq:sh6} it is more complicated. We define $\thij(a, b) = [\tfij(a), \tfji(b)]$ and following the steps of Lemma \ref{L:ident} we can check that for distinct $i, j, k, l$,
\begin{align*}
{}&\thij(a, b) = -(-1)^{(|i|+|j|+|a|)(|i|+|j|+|b|)}\thji(b, a), \\
{}&[\thij(a, b), \tfik(c)] = \tfik(abc), \\
{}&[\thij(a, b), \tfki(c)] = -(-1)^{(|a| + |b|)(|i| + |k| + |c|)} \tfki(cab), \\
{}&[\thij(a, b), \tfkj(c)] = (-1)^{(|i|+|j|+|a|)(|i|+|j|+|b|) + (|a|+|b|)(|j|+|k|+|c|)}\tfkj(cba), \\
{}&[\thij(a, b), \tfij(c)] = \tfij(abc + (-1)^{(|i| + |j| + |a||b| + |b||c| + |c||a|)}cba), \\
{}&[\thij(a, b), \tfkl(c)] = 0.
\end{align*}

When $i, j, k, l$ are distinct we denote
\[
[\tfij(a), \tfkl(1)] = v_{ijkl}(a),
\]
where $v_{ijkl}(a) \in \VV$.  We want to define the homomorphism from $\WW$ by the expression $\rho(\epsilon_{\theta\big((i, j, k, l)\big)}(\overline{ab})) = v_{ijkl}(ab)$, since
\begin{align*}
[\tfij(a), \tfkl(b)] &= [\rho\big(\Fij\sh(a)\big), \rho\big(\Fkl\sh(b)\big)]  \\
{}&= \rho\big([\Fij\sh(a), \Fkl\sh(b)]) = \rho(\epsilon_{\theta\big((i, j, k, l)\big)}(\overline{ab})) = v_{ijkl}(ab).
\end{align*}
Thus, we have to check that
\begin{enumerate}
\item[(R1)] $2v_{ijkl}(a) = 0$,
\item[(R2)] $v_{ijkl}(a) = v_{kjil}(a) = v_{ilkj}(a) = v_{klij}(a)$,
\item[(R3)] $v_{ijkl}(a[b,c]) = 0$,
\item[(R4)] $[\tfij(a), \tfkl(b)] = v_{ijkl}(ab)$.
\end{enumerate}

Suppose that $|i|+|j| = \cero$,
\begin{align*}
0 &=  \big[ \thij(a, b), [\tfij(c), \tfkl(1)] \big]  \\
{}&= \big[ [\thij(a, b), \tfij(c)], \tfkl(1) \big] - \big[\tfij(a), [\thij(1, 1), \tfkl(1)]\big] \\
{}&= [\tfij(abc + (-1)^{|i| + |j| + |a||b| + |b||c| + |c||a|}cba), \tfkl(1)] \\
{}&= v_{ijkl}(abc + (-1)^{|a||b| + |b||c| + |c||a|}cba).
\end{align*}
If $b = c = 1$, we have that
\[v_{ijkl}(2a) = 2v_{ijkl}(a) = 0,
\]
proving (R1).

If $c = 1$, we have that
\[
v_{ijkl}(ab - \menosuno{a}{b}ba) = 0,
\]
so
\[
0 = v_{ijkl}(abc + (-1)^{|a||b| + |b||c| + |c||a|}cba) = v_{ijkl}\big((ab + \menosuno{a}{b}ba)c\big),
\]
implying (R2). If $|k|+|l| = \cero$, the calculation is the same.

On the other hand,
\begin{align*}
[\tfij(a), \tfkl(b)] &= \big[ [\tfik(a), \tfkj(1)], \tfkl(b) \big] \\
{}&= \big[\tfik(a), [\tfkj(1), \tfkl(b)]\big] \\
{}&\quad \ -(-1)^{(|i|+|k|+|a|)(|k|+|j|)}\big[\tfkj(1), [\tfik(a), \tfkl(b)]\big] \\
{}&= (-1)^{(|l|+|k|+|b|)(|k|+|j|)}[\tfil(ab), \tfkj(1)] \\
{}&= v_{ilkj}(ab),
\end{align*}
since the sign does not play any role. Choosing $b = 1$ and using (R1), we have that
\[
v_{ijkl}(a) = v_{ilkj}(a).
\]
Doing the same but changing the indexes we have relations (R3) and (R4).

Thus, the morphism $\rho \colon \sttu\sh \to \widetilde{\mathfrak{st}}(3, 1, A)$ where $\rho\big(\Fij\sh(a)\big) = \tfij(a)$ and $\rho(\epsilon_{\theta\big((i, j, k, l)\big)}(\overline{ab})) = v_{ijkl}(ab)$ is actually a Lie superalgebra homomorphism completing the proof.
\end{proof}

\begin{Co}
The universal central extension of $\mathfrak{sl}(3, 1, A)$ is $\sttu\sh$. Moreover, $\Ho_2\big(\sttu\big) = \WW$.
\end{Co}

\section{Universal central extension of $\stdd$}
In this section we will find the universal central extension of $\stdd$. As in the previous section we consider the partition of $S_4$ but with a small difference. We care now about the name of some cosets, where
\[
\{ (1, 3, 2, 4), (1, 4, 2, 3), (2, 3, 1, 4), (2, 4, 1, 3) \},
\]
will be $P_5$ and
\[
\{ (3, 1, 4, 2), (3, 2, 4, 1), (4, 1, 3, 2), (4, 2, 3, 1) \},
\]
will be $P_6$. The order of the other cosets $P_1, \dots, P_4$ will not make any difference. Note that all the elements of $P_5$ and $P_6$ have the property that $|i| = |k|, |j| = |l|$ and $|i|+|j| = |k|+|l| = \uno$.

We also define a map $\sigma \colon S_4 \to \{ -1, 1 \}$, where $\sigma\big((i, j, k, l)\big) = 1$ if $(i, j, k, l) \in P_1, P_2, P_3$ or $P_4$. In $P_5$,
\begin{align*}
\sigma\big((i, j, k, l)\big) &= 1 \quad &\text{ if } (i, j, k, l) = (1, 3, 2, 4) \text{ or } (2, 4, 1, 3), \\
\sigma\big((i, j, k, l)\big) &= -1 \quad &\text { if } (i, j, k, l) = (1, 4, 2, 3) \text{ or } (2, 3, 1, 4),
\end{align*}
and in $P_6$,
\begin{align*}
\sigma\big((i, j, k, l)\big) &= 1 \quad &\text{ if } (i, j, k, l) = (3, 1, 4, 2) \text{ or } (4, 2, 3, 1), \\
\sigma\big((i, j, k, l)\big) &= -1 \quad &\text { if } (i, j, k, l) = (3, 2, 4, 1) \text{ or } (4, 1, 3, 2).
\end{align*}

On the other hand, let $\WW = A_2^4 \oplus A_0^2$ be $K$-supermodule formed by the direct sum of four copies of $A_2$ and two copies of $A_0$. As in Section \ref{S:sttu} we consider the maps $\epsilon_m(\bar{a}) = (0, \dots, \bar{a}, \dots, 0)$ in position $m$.

We define now the super 2-cocycle which will generate the universal central extension. Using the decomposition of Lemma \ref{L:dec} we build a $K$-bilinear map, defining it in the elements of the $K$-basis and extending it by $K$-linearity,
\[
\psi \colon \stdd \times \stdd \to \WW,
\]
where
\begin{align*}
{}&\psi\big(\Fij(a), \Fkl(b)\big) = \epsilon_{\theta\big((i, j, k, l)\big)}(\overline{ab}), &{}&\text{if } (i, j, k, l) \in P_1,P_2,P_3,P_4 \\
{}&\psi\big(\Fij(a), \Fkl(b)\big) = (-1)^{|b|}\sigma\big((i, j, k, l)\big)\epsilon_{\theta\big((i, j, k, l)\big)}(\overline{ab}), \quad &{}&\text{if } (i, j, k, l) \in P_5 \text{ or } P_6, \\
{}&\psi(x, y) = 0 &{}&\text{if } x \text{ or } y \text{ belong to } \mathcal{H}.
\end{align*}

\begin{Le}
The $K$-bilinear map $\psi$ is a super 2-cocycle.
\end{Le}

\begin{proof}
The map is even since $|i| + |j| + |k| + |l| = \cero$. To check antisymmetry, it suffices to see what happen when $(i, j, k, l) \in P_5$ or $P_6$ since in the other cases the sign do not make any difference since $A_2$ and $A_0$ are commutative. If we suppose that $(i, j, k, l) \in P_5$, we know that $|i| + |j| = |k| + |l| = \uno$,

\begin{align*}
-(-1)^{|\Fij(a)||\Fkl(b)|} \psi\big(\Fkl(b), \Fij(a)\big) &= -(-1)^{(|i|+ |j|+|a|)(|k|+|l|+|b|)}\psi\big(\Fkl(b), \Fij(a)\big) \\
{}&=-(-1)^{(\uno + |a|)(\uno + |b|)}(-1)^{|a|}\sigma\big((k, l, i, j)\big) \epsilon_5(\overline{ba})  \\
{}&=(-1)^{|b|+ |a||b|}\sigma\big((k, l, i, j)\big) \epsilon_5((-1)^{|a||b|}\overline{ab}) \\
{}&=(-1)^{|b|}\sigma\big((i, j, k, l)\big)\epsilon_{\theta\big((i, j, k, l)\big)}(\overline{ab}) \\
{}&=\psi\big(\Fij(a), \Fkl(b)\big),
\end{align*}
since $\sigma\big((i, j, k, l)\big) = \sigma\big((k, l, i, j)\big)$ and $\overline{ab} = (-1)^{|a||b|}\overline{ba}$. If $(i, j, k, l) \in P_6$ it is analogue.

The identity $\psi(x\scero, x\scero) = 0$ where $x\scero \in \big(\stdd\big)\scero$ is straightforward by definition. The last step is to check Jacobi identity. With the aim of simplify notation, we denote by $J(x, y, z)$ the expression

\[
\menosuno{x}{z} \psi([x, y], z) + \menosuno{x}{y} \psi([y, z], x) + \menosuno{y}{z}\psi([z, x], y) = 0,
\]
so we have to check that $J(x, y, z) = 0$ for all $x, y, z \in \stdd$.

Suppose that $\psi([x, y], z) \neq 0$. Using the decomposition of Lemma \ref{L:dec} we see that at most one of $x, y$ belongs to $\mathcal{H}$. Suppose that $x \in \mathcal{H}$. To exclude trivial cases we need that $y = \Fij(a)$ and $z = \Fkl(b)$, where $i, j, k, l$ are distinct. If $(i, j, k, l) \in P_1, \dots, P_4$ the signs does not make any difference so the proof is the same as in \cite[Lemma 2.2]{GaSh}. Let us suppose that $(i, j, k, l) = (1, 3, 2, 4) \in P_5$, the other cases are similar.

If $x = h(c, d)$, then
\begin{align*}
J(x, y , z) &= (-1)^{(|c|+|d|)(|b|+\uno)}\psi\big([h(c, d), F_{13}(a)], F_{24}(b)\big) \\
{}&\quad \ + (-1)^{(|a| + \uno)(|b|+\uno)}\psi\big([F_{24}, h(c,d)], F_{13}(b)\big) \\
{}&= (-1)^{(|c|+|d|)(|b|+\uno)}\psi\big(F_{13}\big( (ab - \menosuno{a}{b}ba)c \big), F_{24}(b)\big) + 0\\
{}&= (-1)^{(|c|+|d|)(|b|+\uno) + |b|} \sigma\big((1, 3, 2, 4)\big) \epsilon_5 \big( \overline{(ab - \menosuno{a}{b}ba)cb}\big) \\
{}&= 0.
\end{align*}

If $x = H_{12}(1, c)$, then
\begin{align*}
J(x, y, z) &= (-1)^{|c|(|a|+\uno)} \psi\big([H_{12}(1, c), F_{13}(a)], F_{24}(b)\big) \\
{}&\quad \ + (-1)^{(|a| + \uno)(|b|+\uno)}\psi\big([F_{24}(b), H_{12}(1, c)], F_{13}(a)\big) \\
{}&=(-1)^{|c|(|a|+\uno)}\psi\big(F_{13}(ca), F_{24}(b)\big) \\
{}&\quad \ +(-1)^{(|a| + \uno)(|b|+\uno) + |c|(|b|+\uno)}\psi\big(F_{24}(cb), F_{13}(a)\big) \\
{}&=(-1)^{|c|(|b|+\uno) + |a|}\sigma\big((1, 3, 2, 4)\big)\epsilon_5(\overline{cab}) \\
{}&\quad \ +(-1)^{(|a|+|c|+\uno)(|b|+\uno) + |b|}\sigma\big((2, 4, 1, 3)\big)\epsilon_5(\overline{cba}) \\
{}&= (-1)^{|c|(|b|+\uno)}\big( (-1)^{|a|}\epsilon_5(\overline{cab}) + (-1)^{(|a|+\uno)(|b|+\uno) + |b|} \epsilon_5(\overline{cba})\big) \\
{}&=(-1)^{|c|(|b|+\uno) + |a|}\big( \epsilon_5(\overline{cab - \menosuno{a}{b}cba}) \big) \\
{}&= 0.
\end{align*}

If $x = H_{13}(1, c)$, then
\begin{align*}
J(x, y, z) &= (-1)^{|c|(|a|+\uno)} \psi\big([H_{13}(1, c), F_{13}(a)], F_{24}(b)\big) \\
{}&= \psi\big( F_{13}(ca + (-1)^{\uno + |a||c| }ac), F_{24}(b) \big) \\
{}&= (-1)^{|b|}\sigma\big((1, 3, 2, 4)\big) \epsilon_5(\overline{(ca - \menosuno{a}{c}ac)b}) \\
{}&= 0.
\end{align*}

If $x = H_{14}(1, c)$, then
\begin{align*}
J(x, y, z) &= (-1)^{|c|(|b|+\uno)} \psi\big([H_{14}(1, c), F_{13}(a)], F_{24}(b)\big) \\
{}&\quad \ + (-1)^{(|a| + \uno)(|b|+\uno)}\psi\big([F_{24}(b), H_{14}(1, c)], F_{13}(a)\big) \\
{}&=(-1)^{|c|(|b|+\uno)}\psi\big(F_{13}(ca), F_{24}(b)\big) \\
{}&\quad \ +(-1)^{(|a| + \uno)(|b|+\uno) + |c|}\psi\big(F_{24}(bc), F_{13}(a)\big) \\
{}&=(-1)^{|c|(|b|+\uno) + |b|}\sigma\big((1, 3, 2, 4)\big)\epsilon_5(\overline{cab}) \\
{}&\quad \ +(-1)^{(|a|+\uno)(|b|+\uno) + |c| + |a|}\sigma\big((2, 4, 1, 3)\big)\epsilon_5(\overline{bca}) \\
{}&= (-1)^{|b|+|c|}\big( (-1)^{|c||b|}\epsilon_5(\overline{cab}) - (-1)^{|a||b|}\epsilon_5(\overline{bca}) \big) \\
{}&=(-1)^{|b|+|c|}\big((-1)^{|c||b|+|b||c|+|a||b|}\epsilon_5(\overline{bca})  - (-1)^{|a||b|}\epsilon_5(\overline{bca})  \big) \\
{}&=0.
\end{align*}

Suppose now that neither $x, y, z \in \mathcal{H}$. If we want that $\psi([x, y], z) \neq 0$ then we must have $\psi\big([\Fik(a), \Fkj(b)], \Fkl(c)\big)$ or $\psi\big([\Fil(a), \Flj(b)], \Fkl(c)\big)$. Again, if $(i, j, k, l) \in P_1, \dots, P_4$ the sign does not matter so the proof is done in \cite{GaSh}. We suppose that $(i, j, k, l) = (1, 3, 2, 4) \in P_5$, since the other cases are similar.

If $x = F_{12}(a)$, $y = F_{23}(b)$ and $z = F_{24}(c)$, then
\begin{align*}
J(x, y, z) &= (-1)^{|a|(|c|+\uno)}\psi\big( F_{13}(ab), F_{24}(c) \big) \\
{}&\quad \ - (-1)^{(|b|+\uno)(|c|+\uno) + |a|(|c|+\uno)}\psi\big( F_{14}(ac), F_{23}(b) \big) \\
{}&=(-1)^{|a|(|c|+\uno) + |c|}\sigma\big((1, 3, 2, 4)\big)\epsilon_5(\overline{abc}) \\
{}&\quad \ -(-1)^{(|a|+|b|+\uno)(|c|+\uno) + |b|}\sigma\big((1, 4, 2, 3)\big)\epsilon_5(\overline{acb}) \\
{}&=(-1)^{|a|(|c|+\uno) + |c|}\big(\epsilon_5(\overline{abc}) + (-1)^{|b||c| + \uno} \epsilon_5(\overline{acb})  \big) \\
{}&=(-1)^{|a|(|c|+\uno) + |c|} \; \epsilon_5(\overline{a(bc - \menosuno{b}{c}cb)}) \\
{}&= 0.
\end{align*}

If $x = F_{14}(a)$, $y = F_{43}(b)$ and $z = F_{24}(c)$, then
\begin{align*}
J(x, y, z) &= (-1)^{(|a|+\uno)(|c|+\uno)}\psi\big( F_{13}(ab), F_{24}(c) \big) \\
{}&\quad \ - (-1)^{|b|(|a|+\uno) + |b|(|c|+\uno)}\psi\big( F_{23}(ac), F_{14}(b) \big) \\
{}&=(-1)^{(|a|+\uno)(|c|+\uno) + |c|}\sigma\big((1, 3, 2, 4)\big)\epsilon_5(\overline{abc}) \\
{}&\quad \ -(-1)^{|b|(|a|+|c|) + |a|}\sigma\big((2, 3, 1, 4)\big)\epsilon_5(\overline{cba}) \\
{}&=-(-1)^{|a|}\epsilon_5\big(\overline{(-1)^{|a||c|}abc - (-1)^{|a||b|+|b||c|} cba}\big) \\
{}&=-(-1)^{|a|}\epsilon_5\big(\overline{(-1)^{|a||c|}abc - (-1)^{|a||b|+|b||c| + |c|(|b|+|a|)} bac}\big) \\
{}&=-(-1)^{|a| + |a||c|}\epsilon_5\big(\overline{(ab - \menosuno{a}{b})c}\big) \\
{}&= 0.
\end{align*}
\end{proof}

We have a central extension
\[
\xymatrix{
0 \ar[r] & \WW \ar[r] & \stdd\sh \ar[r]^{\pi} & \stdd \ar[r] & 0,
}
\]
where $\stdd\sh = \stdd \oplus \WW$ is the Lie superalgebra constructed by the surjective super 2-cocycle $\psi$, defined by the following relations
\begin{align}
{}& a \mapsto \Fij\sh(a) \text{ is a }K\text{-linear map,} \label{Eq:dsh1} \\
{}& [\WW, \WW] = [\Fij\sh(a), \WW] = 0, \label{Eq:dsh2} \\
{}& [\Fij\sh(a), \Fjk\sh(b)] = \Fik\sh(ab) \text{ for distinct i, j, k,} \label{Eq:dsh3} \\
{}& [\Fij\sh(a), \Fij\sh(a)] = 0, \label{Eq:dsh4} \\
{}& [\Fij\sh(a), \Fik\sh(b)] = 0, \label{Eq:dsh5} \\
{}& [\Fij\sh(a), \Fkl\sh(b)] = \epsilon_{\theta\big((i, j, k, l)\big)}(\overline{ab}) \text{ if } (i, j, k, l)\in P_1, P_2, P_3, P_4  \label{Eq:dsh6} \\
{}& [\Fij\sh(a), \Fkl\sh(b)] = (-1)^{|b|}\sigma\big((i, j, k, l)\big)\epsilon_{\theta\big((i, j, k, l)\big)}(\overline{ab}) \text{ if } (i, j, k, l)\in P_5, P_6.  \label{Eq:dsh7}
\end{align}

\begin{Th}\label{T:dd}
The central extension $0 \to \WW \to \stdd\sh \to \stdd$ is universal.
\end{Th}

\begin{proof}
Let
\[
\xymatrix{
0 \ar[r] & \VV \ar[r] & \widetilde{\mathfrak{st}}(2, 2, A) \ar[r]^{\tau} & \stdd \ar[r] & 0
}
\]
be a central extension. As in Theorem \ref{T:du} we have to build a Lie superalgebra homomorphism $\rho \colon \stdd\sh \to \widetilde{\mathfrak{st}}(2, 2, A)$ such that $\tau \circ \rho = \pi$. Choosing preimages $\tfij(a)$ of $\tau$, we have to check they follow relations \eqref{Eq:dsh1}--\eqref{Eq:dsh7}. Doing the analogue computations as in Theorem \ref{T:du} it is obvious that relations \eqref{Eq:dsh1}--\eqref{Eq:dsh5} are satisfied. We have to check that the $\tfij(a)$ follow \eqref{Eq:dsh6} and \eqref{Eq:dsh7} to complete the proof.

As in the previous section, when $i, j, k, l$ are distinct we denote
\[
[\tfij(a), \tfkl(1)] = v_{ijkl}(a).
\]

To satisfy relations \eqref{Eq:dsh6} and \eqref{Eq:dsh7} we want to define the homomorphism from $\WW$ by the expression $\rho(\epsilon_{\theta\big((i, j, k, l)\big)}(\overline{ab})) = \sigma\big((i, j, k, l)\big)v_{ijkl}(ab)$. If $(i, j, k, l) \in P_1, \dots, P_4$ we have to check the conditions
\begin{enumerate}
\item[($\mathcal{R}$1)] $2v_{ijkl}(a) = 0$,
\item[($\mathcal{R}$2)] $v_{ijkl}(a) = v_{kjil}(a) = v_{ilkj}(a) = v_{klij}(a)$,
\item[($\mathcal{R}$3)] $v_{ijkl}(a[b,c]) = 0$,
\item[($\mathcal{R}$4)] $[\tfij(a), \tfkl(b)] = v_{ijkl}(ab)$.
\end{enumerate}

Note that every permutation in $P_1, \dots, P_4$ has an element such that $|i|+|j| = \cero$. Thus, recovering some computations of the previous section we have that
\begin{align*}
0 &= \big[ \thij(1, 1), [\tfij(a), \tfkl(1)] \big] \\
{}&= [\tfij( a+ (-1)^{|i|+|j|}a), \tfkl(1)] \\
{}&= v_{ijkl}(a + (-1)^{|i|+|j|}a).
\end{align*}
If $|i|+|j| = \cero$, we have that $[\tfij(a), \tfkl(1)] = -[\tfij(a), \tfkl(1)]$. Then,
\[
[\tfil(a), \tfkj(1)] = -(-1)^{(|k| + |l|)(|k|+|j|+|b|)} [\tfij(a), \tfkl(1)],
\]
so $[\tfil(a), \tfkj(1)] = -[\tfil(a), \tfkj(1)]$. Doing the same but changing the indexes we obtain ($\mathcal{R}$1) and ($\mathcal{R}$2), and proceeding as in the proof of Theorem \ref{T:sttu}, conditions ($\mathcal{R}$3) and ($\mathcal{R}$4) are satisfied.

If $(i, j, k, l) \in P_5, P_6$, we have that
\begin{align*}
[\tfij(a), \tfkl(b)] &= [\rho\big(\Fij\sh(a)\big), \rho\big(\Fkl\sh(b)\big)]  \\
{}&= \rho\big([\Fij\sh(a), \Fkl\sh(b)]) \\
{}&= \rho\big((-1)^{|b|}\sigma\big((i, j, k, l)\big)\epsilon_{\theta\big((i, j, k, l)\big)}(\overline{ab})\big) \\
{}&=(-1)^{|b|} v_{ijkl}(ab) \\
{}&=(-1)^{|b|}[\tfij(ab), \tfkl(1)].
\end{align*}

Thus, we have to check the following conditions:
\begin{enumerate}
\item[(C1)] $v_{ijkl}(a) = -v_{kjil}(a) =- v_{ilkj}(a) = v_{klij}(a)$,
\item[(C2)] $[\tfij(a), \tfkl(b)] = (-1)^{|b|}[\tfij(ab), \tfkl(1)]$,
\item[(C3)] $v_{ijkl}(a[b,c]) = 0$.
\end{enumerate}

To see (C1),
\begin{align*}
v_{ijkl}(a) &=[\tfij(a), \tfkl(1)] \\
{}& =\big[ \tfij(a), [\tfki(1), \tfil(1)] \big] = \big[ [\tfij(a), \tfki(1)], \tfil(1) \big] \\
{}&=-(-1)^{(|i|+|j|+|a|)(|k|+|i|)}[\tfkj(a), \tfil(1)]\\
{}&=-[\tfkj(a), \tfil(1)] =-v_{kjil}(a),
\end{align*}

and

\begin{align*}
v_{ijkl}(a) &=[\tfij(a), \tfkl(1)] \\
{}& =\big[ \tfij(a), [\tfkj(1), \tfjl(1)] \big] = \\
{}&=(-1)^{(|i|+|j|+|a|)(|k|+|j|)} \big[ \tfkj(1), [\tfij(a), \tfjl(1)] \big]\\
{}&=(-1)^{(|i|+|j|+|a|)(|k|+|j|)}[\tfkj(1), \tfil(a)] = \\
{}&=-(-1)^{(|k|+|j|)(|l|+|j|)}[\tfil(a), \tfkj(1)] = \\
{}&=-[\tfil(a), \tfkj(1)] = -v_{ilkj}(a).
\end{align*}

To see (C2),
\begin{align*}
[\tfij(a), \tfkl(b)] &= \big[ \tfij(a), [\tfkj(1), \tfjl(b)] \big] \\
{}&= (-1)^{(|i|+|j|+|a|)(|k|+|j|)}\big[ \tfkj(1), [\tfij(a), \tfjl(b)] \big] \\
{}&=(-1)^{(|i|+|j|+|a|)(|k|+|j|)}\big[ \tfkj(1), \tfil(ab) \big] \\
{}&\quad \ -(-1)^{(|j|+|l|+|b|)(|k|+|j|)}[\tfil(ab), \tfkj(1)] \\
{}&=-(-1)^{|b|}[\tfil(ab), \tfkj(1)] = (-1)^{|b|}[\tfij(ab), \tfkl(1)] \\
{}&= (-1)^{|b|}v_{ijkl}(ab),
\end{align*}
by part (C1).

Using (C2) and the fact that $|k|+|l| = \uno$,

\begin{align*}
v_{ijkl}(a[b,c]) &= [\tfij(a[b, c]), \tfkl(1)] \\
{}&= (-1)^{|b|+|c|}\sigma\big((i, j, k, l)\big)[\tfij(a), \tfkl(bc - \menosuno{b}{c}cb)] \\
{}&=(-1)^{|b|+|c|}\sigma\big((i, j, k, l)\big)\big[ \tfij(a), [\thkl(b, c), \tfkl(1)] \big] \\
{}&=0,
\end{align*}
by Jacobi identity, we have that (C3) is satisfied.

Thus, we have a Lie superalgebra homomorphism $\rho \colon \stdd\sh \to \widetilde{\mathfrak{st}}(2, 2, A)$ completing the proof.
\end{proof}

\section{Concluding remarks}

Combining the main theorems presented here with the main theorems of \cite{ChSu} and \cite{ChGu} we have a complete characterization of $\Ho_2\big(\stmn\big)$ and $\Ho_2\big(\slmn\big)$ for $m + n \geq 3$.

\begin{Th}\label{T:con1}
Let $K$ a unital commutative ring and $A$ an associative unital $K$-superalgebra with a $K$-basis containing the identity. Then,
\[
\Ho_2\big(\stmn\big)=
\begin{cases*}
0 & $\text{for } m+n  \geq 5 \text{ or }m=2, n=1$, \\
A_3^6 & $\text{for }m=3, n=0$, \\
A_2^6 & $\text{for }m=4, n=0$, \\
\Pi(A_2)^6 & $\text{for }m=3, n=1$, \\
A_2^4 \oplus A_0^2 & $\text{for }m=2, n=2$,
\end{cases*}
\]
where $A_m$ is the quotient of $A$ by the ideal $mA + A[A, A]$ (Definition \ref{D:Am}) and $\Pi$ is the parity change functor.
\end{Th}

\begin{Th}\label{T:con2}
Let $K$ a unital commutative ring and $A$ an associative unital $K$-superalgebra with a $K$-basis containing the identity. Then,
\[
\Ho_2\big(\slmn\big)=
\begin{cases*}
\hc_1(A) & $\text{for } m+n  \geq 5 \text{ or }m=2, n=1$, \\
\hc_1(A) \oplus A_3^6 & $\text{for }m=3, n=0$, \\
\hc_1(A) \oplus A_2^6 & $\text{for }m=4, n=0$, \\
\hc_1(A) \oplus \Pi(A_2)^6 & $\text{for }m=3, n=1$, \\
\hc_1(A) \oplus A_2^4 \oplus A_0^2 & $\text{for }m=2, n=2$,
\end{cases*}
\]
where $A_m$ is the quotient of $A$ by the ideal $mA + A[A, A]$ (Definition \ref{D:Am}) and $\Pi$ is the parity change functor.
\end{Th}


\begin{thebibliography}{99}

\bibitem{Bla}
P.~A. Blaga.
\newblock Dennis supertrace and the {H}ochschild homology of supermatrices.
\newblock {\em Russ. J. Math. Phys.}, 16(3):384--390, 2009.

\bibitem{Blo}
S.~Bloch.
\newblock The dilogarithm and extensions of {L}ie algebras.
\newblock In {\em Algebraic {$K$}-theory, {E}vanston 1980 ({P}roc. {C}onf.,
  {N}orthwestern {U}niv., {E}vanston, {I}ll., 1980)}, volume 854 of {\em
  Lecture Notes in Math.}, pages 1--23. Springer, Berlin-New York, 1981.

\bibitem{ChGu}
H.~Chen and N.~Guay.
\newblock Central extensions of matrix {L}ie superalgebras over
  {$\Bbb{Z}/2\Bbb{Z}$}-graded algebras.
\newblock {\em Algebr. Represent. Theory}, 16(2):591--604, 2013.

\bibitem{ChSu}
H.~Chen and J.~Sun.
\newblock Universal central extensions of $sl_{m|n}$ over
  {$\Bbb{Z}/2\Bbb{Z}$}-graded algebras.
\newblock arXiv:1311.7079, 2013.

\bibitem{Ell1}
G.~J. Ellis.
\newblock A nonabelian tensor product of {L}ie algebras.
\newblock {\em Glasgow Math. J.}, 33(1):101--120, 1991.

\bibitem{Gao}
Y.~Gao.
\newblock Steinberg unitary {L}ie algebras and skew-dihedral homology.
\newblock {\em J. Algebra}, 179(1):261--304, 1996.

\bibitem{GaSh}
Y.~Gao and S.~Shang.
\newblock Universal coverings of {S}teinberg {L}ie algebras of small
  characteristic.
\newblock {\em J. Algebra}, 311(1):216--230, 2007.

\bibitem{Gar}
H.~Garland.
\newblock The arithmetic theory of loop groups.
\newblock {\em Inst. Hautes \'Etudes Sci. Publ. Math.}, 52:5--136, 1980.

\bibitem{IoKo1}
K.~Iohara and Y.~Koga.
\newblock Central extensions of {L}ie superalgebras.
\newblock {\em Comment. Math. Helv.}, 76(1):110--154, 2001.

\bibitem{KaLo}
C.~Kassel and J.-L. Loday.
\newblock Extensions centrales d'alg\`ebres de {L}ie.
\newblock {\em Ann. Inst. Fourier (Grenoble)}, 32(4):119--142 (1983), 1982.

\bibitem{MiPi}
A.~V. Mikhalev and I.~A. Pinchuk.
\newblock Universal central extensions of the matrix {L}ie superalgebras {${\rm
  sl}(m,n,A)$}.
\newblock In {\em Combinatorial and computational algebra ({H}ong {K}ong,
  1999)}, volume 264 of {\em Contemp. Math.}, pages 111--125. Amer. Math. Soc.,
  Providence, RI, 2000.

\bibitem{Neh}
E.~Neher.
\newblock An introduction to universal central extensions of {L}ie
  superalgebras.
\newblock In {\em Groups, rings, {L}ie and {H}opf algebras ({S}t. {J}ohn's,
  {NF}, 2001)}, volume 555 of {\em Math. Appl.}, pages 141--166. Kluwer Acad.
  Publ., Dordrecht, 2003.

\bibitem{SCG}
S.~Shang, H.~Chen, and Y.~Gao.
\newblock Central extensions of {S}teinberg {L}ie superalgebras of small rank.
\newblock {\em Comm. Algebra}, 35(12):4225--4244, 2007.

\bibitem{Van}
W.~L.~J. van~der Kallen.
\newblock {\em Infinitesimally central extensions of {C}hevalley groups}.
\newblock Lecture Notes in Mathematics, Vol. 356. Springer-Verlag, Berlin-New
  York, 1973.

\end{thebibliography}

\end{document}